\documentclass[12pt,a4paper,openbib]{article}
\usepackage[utf8]{inputenc}
\usepackage[T1]{fontenc}
\usepackage[french]{babel}
\usepackage{amsmath}
\usepackage{amsfonts}
\usepackage{dsfont}
\usepackage{tikz}
\usetikzlibrary{cd}
\usepackage{authblk}
\usepackage{amssymb}
\usepackage{amsthm}
\usepackage{lmodern}
\usepackage{fancyhdr}
\usepackage{lettrine}
\usepackage{bm}
\usepackage{enumerate}
\usepackage{multicol}
\usepackage[hidelinks]{hyperref}
\usepackage[nottoc, notlof, notlot]{tocbibind}
\usepackage[left=2cm,right=2cm,top=2cm,bottom=2cm]{geometry} 
\usepackage{graphicx}
\usepackage{mathtools}
\usepackage{wasysym}
\usepackage{stmaryrd}
\usepackage[all,cmtip]{xy}
\usepackage{enumitem}
\usepackage{relsize} 
\usepackage{mathrsfs}

\makeatletter

\newcommand{\address}[1]{\gdef\@address{#1}}
\newcommand{\email}[1]{\gdef\@email{\url{#1}}}
\newcommand{\@endstuff}{\par\vspace{\baselineskip}\noindent\small
\begin{tabular}{@{}l}\scshape\@address\\\textit{E-mail address:} \@email\end{tabular}}
\AtEndDocument{\@endstuff}
\makeatother

\author{Damien Junger\footnote{This work has been written in a great part during the author PhD thesis at the ENS Lyon. His work are currently funded by the Deutsche Forschungsgemeinschaft (DFG, German Research Foundation) under Germany's Excellence Strategy EXC 2044 –390685587, Mathematics Münster: Dynamics–Geometry–Structure.}}
\address{Mathematisches Institut, Universität Münster,\\ Fachbereich Mathematik und Informatik der Universität Münster,  Orléans-Ring 10, 48149 Münster, Germany.}
\email{djunger@uni-muenster.de}
\title{Un autre calcul des fonctions inversibles sur l'espace symétrique de Drinfeld}

\newtheorem{theointro}{Th\'eor\`eme}

\newtheorem{theo}{Th\'eor\`eme}[section]
\newtheorem{lem}[theo]{Lemme}
\newtheorem{coro}[theo]{Corollaire}
\newtheorem{prop}[theo]{Proposition}
\newtheorem{propintro}[theointro]{Proposition}

\theoremstyle{definition}

\theoremstyle{remark}
\newtheorem{rem}[theo]{Remarque}
\newtheorem{remintro}{Remarque}

\newtheorem{ex}[theo]{Exemple}

\DeclareMathOperator{\spec}{Spec}
\DeclareMathOperator{\spf}{Spf}
\DeclareMathOperator{\spg}{Sp}

\DeclareMathOperator{\gln}{GL}
\DeclareMathOperator{\pgln}{PGL}

\DeclareMathOperator{\modut}{-Mod}

\DeclareMathOperator{\ens}{Ens}

\DeclareMathOperator{\hhh}{H}

\DeclareMathOperator{\homm}{Hom}

\DeclareMathOperator{\pic}{Pic}

 \newcommand{\iso}{\stackrel{\sim}{\fl}}

\font\tengoth=eufb10
\font\sevengoth=eufb7
\font\fivegoth=eufb5

\newfam\gothfam
\textfont\gothfam=\tengoth
\scriptfont\gothfam=\sevengoth
\scriptscriptfont\gothfam=\fivegoth

\def\B{{\mathbb{B}}}

\def\F{{\mathbb{F}}}
\def\G{{\mathbb{G}}}
\def\H{{\mathbb{H}}}

\def\N{{\mathbb{N}}}

\def\P{{\mathbb{P}}}
\def\Q{{\mathbb{Q}}}

\def\Z{{\mathbb{Z}}}

\def\BC{{\mathcal{B}}}
\def\CC{{\mathcal{C}}}

\def\HC{{\mathcal{H}}}

\def\GC{{\mathcal{G}}}

\def\OC{{\mathcal{O}}}

\def\TC{{\mathcal{T}}}

\def\mG{{\mathfrak{m}}}

\def\Of{{\mathscr{O}}}

\def\bar#1{\overline{#1}}
\def\het#1{{\rm H}^{#1}_{\rm ét}}

\def\hdr#1{{\rm H}^{#1}_{\rm dR}}

\def\hgal#1{{\rm H}^{#1}_{\rm Gal}}

\def\et{\text{ et }}
\def\si{\text{Si }}
\def\sinon{\text{Sinon }}
\def\and{\text{ and }}

\def\fl{\rightarrow}

\def\fln#1#2{\xrightarrow[#2]{#1}}

\def\flinj{\hookrightarrow}

\def\fla{\mapsto}

\def\limp{\varprojlim}

\def\som#1#2#3{\sum\limits_{{\substack{#2}}}^{#3}{#1}}
\def\pro#1#2#3{\prod\limits_{{\substack{#2}}}^{#3}{#1}}

\def\uni#1#2#3{\bigcup\limits_{{\substack{#2}}}^{#3}{#1}}

\def\drt#1#2#3{\bigoplus\limits_{{\substack{#2}}}^{#3}{#1}}

\begin{document}

\maketitle

\begin{abstract}
In this article, we give an explicit description of the invertible functions on the Drinfeld symmetric space over $K$ a finite extension of $\mathbb{Q}_p$. We identify them with some distribution spaces over the profinite set of $K$-rationnal points of the projective space. The strategy consists of constructing a map from these distributions to the invertible functions following the methods of Schneider-Stuhler, Iovita-Spiess, de Shalit. We show that it is compatible with the isomorphisms they constructed to compute \'etale and de Rham cohomology in degree $1$ and that this property forces our desired map to be an isomorphism. In particular, we get a $\mathbb{Z}$-structure on these cohomology groups.
\end{abstract}

\tableofcontents
 
\section*{Introduction}

Soit $p$ un nombre premier, et $K$ une extension finie de $\Q_p$. Le but de ce travail est d'étudier et de calculer les fonctions inversibles sur l'espace symétrique de Drinfeld $\H^d_K$ sur $K$ de dimension $d$. Plus précisément, nous obtenons une description explicite de $\Of^*(\H^d_K)/K^*$ (où $\Of^*$ désigne le faisceau des fonctions inversibles) en tant que représentation de $G:=\gln_{d+1} (K)$ et cela constitue le résultat principal de cet article. 

 Ce dernier peut être vu comme un complément à l'article \cite{J1} où nous avons pu calculer toute la cohomologie analytique à coefficient dans $\G_m$ de $\H^d_K$ et nous proposons ici une preuve alternative du cas particulier de degré $0$. Notons que le calcul de $\Of^*(\H^d_K)/K^*$ a aussi fait l'objet de l'article\footnote{ $\bullet$ Pour une description de $\Of^*(\H^d_K)$ en tant qu'extension de $K^*$ et de $\Of^*(\H^d_K)/K^*$, cf. \cite{geh}.
 
 $\bullet$ Voir aussi les travaux de Van Der Put en dimension $1$ \cite{vdp2} ou encore \cite[I.8.9]{frvdp}.}
 \cite{gek}.  Là où cet article se distingue des deux travaux précédemment cités, c'est dans la stratégie et les méthodes employées. En effet, le groupe $\Of^*(X)$ encode des informations cohomologiques importantes  pour un espace rigide lisse $X$. Nous avons par exemple deux applications $d\log : u\in\Of^*(X)\fla du/u\in \hdr{1}(X)$, $\kappa:\Of^*(X)\to \het{1}(X,\Q_l)$ et nous pouvons contrôler le noyau et le conoyau de la dernière  grâce à des invariants fondamentaux de $X$ comme le groupe de Picard via la suite exacte de Kummer. Nous utilisons ici de manière cruciale le calcul des groupes $\hdr{*}(\H^d_K)$, $\het{*}(\H^d_K,\Q_l)$ réalisé entre autres par les travaux précurseurs \cite{scst, iovsp, ds}. En particulier, nous montrons que les représentations décrites dans ces articles déterminent entièrement la structure de $\Of^*(\H^d_K)/K^*$. Ainsi, nous pouvons voir le résultat principal comme une forme de réciproque à la proposition 8.1. de \cite{J1} où l'on retrouve les résultats de \cite{scst} sur $\het{1}(\H^d_K,\Q_l)$ grâce aux calculs de  $\Of^*(\H^d_K)/K^*$ et de  $\pic(\H^d_K)$.

Si $X=Y\backslash D$ est un schéma algébrique, complémentaire d'un diviseur à croisements normaux $D$, le groupe $\Of^*(X)$ a aussi un lien profond avec la géométrie  d'une compactification $Y$. Par exemple, toute fonction méromorphe sur $Y$ dont les pôles et les zéros sont concentrés sur $D$ donne lieu à une fonction inversibles sur $X$. En particulier, on peut étudier les composantes irréductibles de $D$ et leur combinatoire à travers $\Of^*(X)$. Par analogie dans le cas des espaces rigides, nous utilisons de manière essentielle dans la preuve l'existence du  plongement de l'espace symétrique dans l'espace projectif $\H^d_K\subset \P^d_K$ dont le complémentaire est  l'union des hyperplans $K$-rationnels (voir la remarque \ref{remintro} pour plus de détails). 

Avant d'énoncer notre résultat principal, nous aimerions donner une dernière motivation à notre calcul. Du fait de l'annulation du groupe de Picard calculé dans \cite[Théorème 7.1.]{J1}, tous les revêtements cycliques d'ordre premier à $p$ sont de type Kummer et leur structure est entièrement déterminée par $\Of^*(\H^d_K)/K^*$. Nous pouvons par exemple décrire explicitement ceux  qui sont de plus $G$-équivariants. Un exemple fondamental est le premier revêtement de la tour de Drinfeld $(\Sigma^n)_{n\in\N}$ construite dans l'article monumental \cite{dr2}, tour de revêtement qui s'est révélée essentielle pour comprendre et étudier  les correspondances de Jacquet-Langlands et de Langlands local $l$-adique. 
Même si la combinatoire et la géométrie des espaces  $\Sigma^n$  semblent  inaccessibles, nous parviendrons pour le revêtement modéré $\Sigma^1$ à donner une équation explicite à cet espace dans un travail futur.

Enonçons le résultat principal de l'article. Soit $A$ un groupe abélien,  introduisons la représentation Steinberg lisse ${\rm St}_1(A)$ à coefficient dans $A$ (nous décrirons un peu plus tard le lien entre cet objet et les calculs réalisés dans \cite{scst}) et notons $M^\vee$ le $A$-dual algébrique $M^\vee:=
\homm_{A}(M),A)$ d'un $A$-module $M$.
\begin{theointro}\label{INVA}

On a un isomorphisme $G$–équivariant :
\[\Of^*(\H^d_K)/K^*\cong{\rm St}_1(\Z)^\vee.\]

\end{theointro}

Pour comprendre dans quel contexte s'inscrit cet énoncé, nous devons expliquer plus précisément  le calcul réalisé par Schneider et Stuhler \cite{scst}. L'un des intérêts essentiels de leurs résultats se trouvent dans la généralité des  méthodes employées. En effet, elles s'appliquent à toute théorie cohomologique sur la catégorie des espaces analytiques lisses, vérifiant certains axiomes abstraits, le plus important étant l'axiome d'homotopie 
(i.e. la projection $X\times \mathring{\B}^1 \rightarrow  X$ induit un isomorphisme $\hhh^*(X) \iso \hhh^*(X\times \mathring{\B}^1)$ avec $\mathring{\B}^1$ la boule unité ouverte de dimension $1$). On impose aussi 
\[
\hhh^i(\P_K^d)=\begin{cases}
A & \si i=2k\in \llbracket 0,2d \rrbracket \\ 
0 &  \sinon
\end{cases}
\]
pour un anneau artinien $A$ fixé. La cohomologie de de Rham resp. étale $l$-adique vérifient ces hypothèses quand $A=K$ resp. $A=\Q_l$.  

Schneider et Stuhler exhibent alors  une suite de complexes simpliciaux explicite $(\TC_\bullet^{(k)})$ pour lesquels on a
 un isomorphisme $G$–équivariant \cite[§3, Theorem 1.]{scst}
\[\hhh^k(\H_{K}^d)\cong{\rm St}_k(A)^\vee\] avec ${\rm St}_k(A):=\tilde\hhh_{\rm Simp}^{k-1} (|\TC^{(k)}_\bullet |, A)$  la cohomologie simpliciale relative de $\TC^{(k)}_\bullet$.

Le calcul précédent montre que les représentations ${\rm St}_k(K)^\vee$ et ${\rm St}_k(\Q_l)^\vee$ sont des invariants cohomologiques pour l'espace symétrique de Drinfeld. Nous aimerions comprendre quelles théories cohomologiques pourraient être décrites par ${\rm St}_k(A)^\vee$ pour un anneau artinien quelconque. Plus généralement, on peut définir pour tout groupe abélien $A$ la représentation Steinberg via la cohomologie du complexe simplicial $\TC^{(k)}_\bullet$. Quelle interprétation donner à l'objet ${\rm St}_k(A)^\vee$? Par exemple quand $A=\Z$, la représentation ${\rm St}_k(\Z)^\vee$ fournit une $\Z$–structure à ${\rm St}_k(K)^\vee$ et à ${\rm St}_k(\Q_l)^\vee$, et donc une $\Z$–structure aux cohomologies $l$–adique et de de Rham de $\H^d_K$.  Le résultat principal de cet article fournit une telle interprétation quand $k=1$.

Un des aspects techniques du calcul de Schneider-Stuhler provient du fait que l'isomorphisme qu'ils ont exhibé est purement abstrait. Les méthodes de Iovita-Spiess \cite{iovsp} permettent de pallier à ce problème en interprétant ${\rm St}_k(A)^\vee$ comme un espace de distribution\footnote{Il s'agit seulement d'un sous-espace de l'espace totale des distributions qui sont non dégénérées au sens de \ref{paragraphstiovsp}. } à valeur dans $A$ sur l'espace profini $(\P^d(K))^{k+1}$ (nous l'avons noté $\tilde{\rm D}_k(A)$ dans la suite.). Fort de ce point de vu, ils construisent des isomorphismes $\beta^{(k)} : {\rm St}_1 (\Z/l\Z)^\vee\iso \het{k}(\H_{K}^d\otimes C, \mu_l^{\otimes k})$ et $\gamma^{(k)} : {\rm St}_1(K)^\vee\iso \hdr{k} (\H_{K}^d)$ avec $C=\hat{\bar{K}}$ le complété d'une clôture algébrique de $K$. Précisons d'où proviennent ces flèches. Pour tout couple d'hyperplan $(H_{ 0},  H_{1})$, on peut définir une fonction inversible $\frac{l_{H_1}}{l_{H_0}}$ unique modulo les constantes, où $l_{H_i} $ est une forme linéaire s'annulant sur $H_{ i}$. On peut alors construire des applications 
$$f : (H_{ 0},\cdots, H_{ k})\in (\P^d(K))^{k+1} \mapsto\bar{\kappa} (\frac{l_{H_1}}{l_{H_0}})\cup \cdots \cup \bar{\kappa} (\frac{l_{H_k}}{l_{H_0}}) \in \het{k}(\H_{K}^d \otimes C, \mu_l^{\otimes k}),$$
  $$g : (H_{ 0},\cdots, H_{k})\in (\P^d(K))^{k+1}\mapsto d\log (\frac{l_{H_1}}{l_{H_0}})\wedge \cdots \wedge d\log (\frac{l_{H_k}}{l_{H_0}}) \in \hdr{k} (\H_{K}^d),$$ 
avec $\bar{\kappa}$ le morphisme de bord dans la suite exacte longue de Kummer et $\cup$ le cup-produit. Ces applications peuvent être  intégrées le long des distributions considérées. Les isomorphismes $\beta^{(k)}$, $\gamma^{(k)}$ recherchés sont  donnés par les intégrales suivantes   
 $$\beta^{(k)}(\mu):=\int_{(H_{i})_i\in (\P^d(K))^{k+1}} f ((H_{i})_i){\rm d} \mu,\ \ \gamma^{(k)}(\mu):=\int_{(H_{i})_i\in (\P^d(K))^{k+1}} g  ((H_{i})_i) {\rm d} \mu.$$ 
  Le point essentiel est de généraliser ces méthodes quand $A=\Z$ ce qui se traduit par ce résultat intermédiaire.

\begin{propintro}[Théorème \ref{theoprincunite}]\label{propintrointdist}

   Il existe un morphisme 
   $$\alpha : {\rm St}_1 (\Z)^\vee \to   \Of^*(\H_{ K}^d)/K^*$$
   uniquement caractérisé par 
   $$\alpha(\delta_{H_0, H_1})=\frac{l_{H_1}}{l_{H_0}}$$
pour tous $H_0,H_1\in \HC$ où $\delta_{H_{0},H_{1}}$ est la masse de Dirac en $(H_{0},H_{1})$. De plus on dispose de diagrammes commutatifs 
 
 \[
\xymatrix{ 
\tilde{\rm D}_{1} (\Z) \ar[r]^-{\alpha} \ar[d] & \Of^*(\H_{K}^d)/K^*\ar[d]^{\bar{\kappa}}  \\
\tilde{\rm D}_{1} (\Z/l\Z) \ar[r]^-{\beta} & \het{1}(\H_{K }^d\otimes C, \mu_l)  }\ \ \ \ \  
\xymatrix{ 
\tilde{\rm D}_{1} (\Z) \ar[r]^-{\alpha} \ar[d]^{} & \Of^*(\H_{K}^d)/K^*\ar[d]^{d\log}  \\
\tilde{\rm D}_{1} (K) \ar[r]^-{\gamma} & \hdr{1} (\H_{K}^d)  }
\]
avec $\beta=\beta^{(1)}$ et $\gamma=\gamma^{(1)}$.

\end{propintro}

\begin{remintro}\label{remintro}

\begin{itemize}
\item Comme dans le calcul de \cite{iovsp}, on peut aussi voir $\alpha$ comme une flèche d'intégration. Montrer que cette dernière est un isomorphisme implique en particulier que le groupe engendré par $\frac{l_{H_1}}{l_{H_0}}$ pour tous $H_0,H_1\in \HC$ est dense dans $\Of^*(\H_{ K}^d)$. Ainsi, toute fonction inversible peut être vu comme une limite de produit fini de formes linéaires homogènes de degré $0$ précisant ainsi le lien entre les fonctions inversibles sur $\H_{ K}^d$ et la compactification $\P_{ K}^d$.
\item Les domaines de périodes  étudier dans \cite{GPW4} peuvent être vu comme une généralisation de l'espace symétrique. Leur géométrie possède des analogies profondes avec celle de $\H_{ K}^d$ et ces espaces admettent des compactifications dont le complémentaire est une union de fermés explicites. De plus leur cohomologie est connue et déterminée dans \cite{orl,GPW4}. La similarité des descriptions des cohomologies obtenues  suggèrent encore l'existence d'une $\Z$-structure naturelle. Nous espérons que les méthodes utilisées dans cet article pourrons aider à les exhiber. 
\end{itemize}

\end{remintro}
 
Le reste de l'argument consiste à prouver que seul un isomorphisme peut vérifier les mêmes compatibilités que $\alpha$ avec $\beta$, $\gamma$. 
Plus précisément, en utilisant la théorie des résidus de de Shalit \cite{ds} nous montrons que 
   l'image de $\alpha$ est un facteur direct. La
compatibilité avec la cohomologie étale géométrique implique que le supplémentaire de l'image de $\alpha$ est $l^\infty$-divisible modulo les constantes. Il s’agira alors d'établir qu’une fonction inversible  $l^\infty$-divisible dans $\Of^*(\H^d_{K})/K^*$  est en fait constante, ce qui se fait en utilisant la géométrie du modèle formel et le lien avec 
l'immeuble de Bruhat-Tits.


\subsection*{Notations et conventions}
 On fixe dans tout ce qui va suivre un nombre premier $p$ et une extension finie $K$ de $\Q_p$, dont on note $\mathcal{O}_K$ l'anneau des entiers, $\varpi$ une uniformisante et $\F=\F_q$ le corps r\'esiduel. On note $C=\hat{\bar{K}}$ le complété d'une clôture algébrique de $K$. 
  Soit $L\subset C$ une extension complète de $K$,  susceptible de varier, d'anneau des entiers $\mathcal{O}_L$, d'idéal maximal  $\mG_L$ et de corps r\'esiduel $\kappa$. 

On note  $\P_{L}^n$ l'espace projectif rigide  de dimension  $n$ sur $L$,  $\B^n_L$  la boule unité et  $\mathring{\B}^n_L$ la boule unité ouverte. Si $X$ est un $L$-espace analytique réduit, on note $\Of^+_X$ le faisceau des fonctions à puissances bornées, $\Of^{++}_X$ le faisceau des fonctions topologiquement nilpotentes,  $\Of^{*}_{X}$ (ou bien $\G_{m,X}$) le faisceau des fonctions inversibles, $\Of^{**}_{X}$ le faisceau $1+ \Of^{++}_X$ et  $\Of^{**}_{m,X}$ le faisceau $1+ \varpi^m\Of^{++}_X$. Pour tout ouvert affinoïde $U\subset X$,  on munit $\Of^*_X (U)$ de la topologie induite par le plongement $\Of^*_X(U)\to\Of_X (U)^2 : f\fla  (f,f^{-1})$ (muni de la norme spectrale). En particulier, $(\Of^{**}_{m,X}(U))_m$ est un base de voisinage de $\Of^*_X (U)$. 

Si $ \Lambda$ est un groupe cyclique d'ordre $N$ premier à  $p$ et $\bar{X}=X\hat{\otimes}_L C$, le morphisme de Kummer  sera noté $\kappa : \Of^* (X)\to  \het{1}(X,\Lambda_X)$ et $\bar{\kappa} :  \Of^* (X)\to  \het{1}(\bar{X},\Lambda_{\bar{X}})$ sera la restriction de $\Of^* (\bar{X})\to  \het{1}(\bar{X},\Lambda_{\bar{X}})$. Si $X$ est lisse, on a aussi un morphisme $d\log :u\in \Of^* (X)\fla d u/u\in  \hdr{1}(X)$.

\subsection*{Remerciements}

Le présent travail a été, avec \cite{J1,J3,J4}, en grande partie réalisé durant ma thèse à l'ENS de lyon,  et a pu bénéficier de nombreux conseils et de l'accompagnement constant de mes deux maîtres de thèse Vincent Pilloni et Gabriel Dospinescu. Je les en remercie très chaleureusement.

\section{Rappels sur la géométrie de l'espace de Drinfeld}

 Nous rappelons quelques résultats standards concernant la géométrie de l'espace symétrique de Drinfeld 
 en renvoyant à  (\cite[section 1]{boca}, \cite[sous-sections I.1. et II.6.]{ds}, \cite[sous-section 3.1.]{dat1}, \cite[sous-sections 2.1. et 2.2]{wa}) pour les détails manquants.
 
 On fixe une extension finie $K$ de $\Q_p$, une uniformisante $\varpi$ de $K$ et un entier $d\geq 1$. On note $\F=\F_q$ le corps résiduel de $K$ et $G={\rm GL}_{d+1}(K)$.  
 \subsection{L'immeuble de Bruhat-Tits} \label{paragraphbtgeosimpstd}

\label{paragraphbtsimp}
Notons $\BC\TC$ l'immeuble de Bruhat-Tits associé au groupe $\pgln_{d+1}(K)$. Le $0$-squelette $\BC\TC_0$ de l'immeuble est l'ensemble des réseaux de $K^{d+1}$ à  homothétie près, i.e. $\BC\TC_0$ s'identifie à  $G/K^*\gln_{d+1}(\OC_K)$. Un $(k+1)$-uplet de sommets 
$\sigma=\{s_0,\cdots, s_k\}\subset \BC\TC_{0} $ est un $k$-simplexe de $\BC\TC_k$ si et seulement si, quitte à  permuter les sommets $s_i$, on peut trouver pour tout $i$ des réseaux $M_i$   avec $s_i=[M_i]$ tels que \[M_0\supsetneq M_1\supsetneq\cdots\supsetneq M_k \supsetneq \varpi M_0.\] La donnée d'un tel ordre sur le sommets est unique à  permutation circulaire près et est déterminée par  le choix du sommet $s_0$ que l'on appellera distingué. Nous appellerons 
$\widehat{\BC\TC}_k$ l'ensemble des $k$-simplexes pointés par un sommet. 


La réalisation topologique de l'immeuble sera notée $|\BC\TC|$. On confondra les simplexes avec leur réalisation topologique de telle manière que $|\BC\TC|=\uni{\sigma}{\sigma\in \BC\TC}{}$. Les différents $k$-simplexes, vus comme des compacts de la réalisation topologique, seront appelés faces. L'intérieur d'une face $\sigma$ sera noté $\mathring{\sigma}=\sigma \backslash \uni{\sigma'}{\sigma'\subsetneq\sigma}{}$ et sera appelé cellule.


 Fixons un simplexe pointé $\sigma\in\widehat{\BC\TC}_{k}$ et considérons une présentation : \[M_0\supsetneq M_1\supsetneq\cdots\supsetneq M_k \supsetneq \varpi M_0=M_{k+1}\] En posant 
 $\bar{M}_i=M_i/\varpi M_0,$
 on obtient un drapeau $$\bar{M}_0\supsetneq \bar{M}_1\supsetneq\cdots\supsetneq \bar{M}_k \supsetneq \bar{M}_{k+1}=0$$ 
dans $\bar{M}_0\cong \F_q^{d+1}$.  On note 
 $$d_i=d+1-{\rm dim}_{\F_q} (\bar{M}_i),\,\, e_i=d_{i+1}-d_{i}.$$
  Nous dirons que le simplexe $\sigma$ est de type $(e_0, e_1,\cdots, e_k)$.
  
  Considérons une base $(\bar{f}_0,\cdots, \bar{f}_d)$ adaptée  au drapeau $(\bar{M}_i)_i$ i.e. telle que $\bar{M}_i=\left\langle \bar{f}_{d_i},\cdots , \bar{f}_{d}\right\rangle$ pour tout $i$. Pour tout choix de relevés $(f_0,\cdots ,f_d)$ de $(\bar{f}_0,\cdots ,\bar{f}_d)$ dans $M_0$, on a
  $$M_i=\left\langle\varpi f_0,\cdots,\varpi f_{d_i -1},  f_{d_i},\cdots,  f_d \right\rangle=\varpi (N_0\oplus \cdots \oplus N_{i-1})\oplus N_i \oplus \cdots\oplus N_k,$$
  où 
  $$N_i=\left\langle f_{d_{i}},\cdots, f_{d_{i+1} -1} \right\rangle.$$
  Si $(f_0,\cdots ,f_d)$ est la base canonique de $K^{d+1}$, nous dirons que $ \sigma $ est le simplexe standard de type $(e_0, e_1,\cdots, e_k)$. En particulier, tout simplexe  de $\BC\TC$ est conjugué sous l'action de $G$ à l'unique  simplexe  standard du même type.






   Considérons la projection naturelle $M_i\setminus M_{i+1}\subset M_i\to (M_i/ \varpi M_i)/\F^*$. On choisit un sous-ensemble $R_i$ de $M_i\setminus M_{i+1}$ qui intersecte chaque fibre de la projection en un point. On fait de même avec $N_i\setminus\varpi N_i$ et la projection $N_i\to (N_i/ \varpi N_i)/\F^*$, obtenant ainsi un sous-ensemble $\tilde{R}_i$ de $N_i\setminus\varpi N_i$.

\subsection{L'espace des hyperplans $K$-rationnels} 

   On note $\HC$ l'ensemble des hyperplans $K$-rationnels dans $\P^d$.  L'ensemble $\mathcal{H}$ est profini car il s'identifie à $\P^d(K)$.






   
 %

Définissons maintenant quelques données relatives à l'ensemble $\HC$.
   Si $a=(a_0,\dots, a_d)\in C^{d+1}\backslash \{0\}$, $l_a$ désignera l'application \[b=(b_0,\dots, b_d)\in C^{d+1} \mapsto \left\langle a,b\right\rangle := \som{a_i b_i}{0\leq i\leq d}{}.\] Ainsi $\HC$ s'identifie à  $\{\ker (l_a),\; a\in K^{d+1}\backslash \{0\} \}$ et à  $\P^d (K)$.
   
   Le vecteur $a=(a_i)_{i}\in C^{d+1}$ est dit unimodulaire si $|a|_{\infty}(:=\max (|a_i|))=1$. L'application 
   $a\mapsto H_a:=\ker (l_a)$ induit une bijection entre le quotient de l'ensemble des vecteurs unimodulaires 
   $a\in K^{d+1}$ par l'action évidente de $\OC_K^*$ et l'ensemble $\mathcal{H}$. 
   
   Pour $a\in K^{d+1}$ unimodulaire et $n\geq 1$ 
   on considère l'application $l_a^{(n)} $ \[b\in (\OC_{C}/\varpi^n)^{d+1}\mapsto \left\langle a,b\right\rangle\in \OC_{C}/\varpi^n\] et on note $$\HC_n =\{\ker (l_a^{(n)}),\; a\in K^{d+1}\backslash \{0\} \; {\rm unimodulaire} \}\simeq\P^d(\OC_K/\varpi^n).$$ Alors $\HC = \varprojlim_n \HC_n$ et chaque 
   $\HC_n$ est fini.


\subsection{Géométrie de l'espace symétrique}
\label{paragraphhdkrectein}

Nous allons maintenant décrire l'espace symétrique de Drinfeld $\H_K^d$. Il s'agit l'ouvert de l'espace projectif de dimension $d$ auquel on a retiré les hyperplan $K$-rationnels i.e.  \[\H_K^d =\P^d_{rig, K}\backslash \uni{H}{H\in \HC}{}.\]

Pour tout $H\in \HC_n$, on choisit un vecteur unimodulaire $a$ tel que $H=\ker(l_a^{(n)}$. On a deux recouvrements admissibles croissants par des ouverts admissibles (ne dépend pas des choix des vecteurs unimodulaires $a$)
$\H^d_K =\uni{\mathring{U}_n}{n> 0}{}=\uni{\bar{U}_n}{n \ge 0}{}$
où 
$$\mathring{U}_n= \{z\in\P^d_{  K} : |\varpi|^{-n} > |\frac{l_a(z)}{l_b(z)}|>|\varpi|^n,\forall a\neq b \in \HC_{n}\}$$
et 
$$ \bar{U}_n= \{z\in\P^d_{  K} : |\varpi|^{-n} \ge |\frac{l_a(z)}{l_b(z)}|\ge |\varpi|^n,\forall a\neq b \in \HC_{n+1}\}
$$ et ces derniers munissent $\H^d_K$ d'une structure naturelle d'espace rigide.

Le recouvrement par les affinoïdes $\bar{U}_n$ est Stein \cite[Proposition 4]{scst}. Le deuxième est cofinal et chaque $ \mathring{U}_n$ est Stein.  
On note $\bar{U}_{n,L}=\bar{U}_{n}\hat{\otimes}_K L$, $\mathring{U}_{n,L}=\mathring{U}_{n} \hat{\otimes}_K L$ et $\H^d_L=\H^d_K\hat{\otimes}_K L$. 

   Il nous sera très utile pour la suite de considérer d'autres recouvrements de $\H^d_K$, reliés à l'application de réduction vers l'immeuble de Bruhat-Tits. 
  On a une application $G$-équivariante \[\tau : \H^d_K(C)\to \{\text{normes sur } K^{d+1}\}/\{\text{homothéties}\} \] donnée par $$\tau (z): v\mapsto |\som{z_i v_i}{i=0}{d}|$$ si $z=[z_0,\cdots, z_d] \in \H_K^d (C)$. L'image $\tau(z)$ ne dépend pas du représentant de $z$ car les normes sont vues à  homothétie près. Le fait de prendre le complémentaire des hyperplans $K$-rationnels assure que $\tau(z)$ est bien séparée et donc  une norme sur $K^{d+1}$.

D’après un résultat classique de Iwahori-Goldmann \cite{iwgo} l'espace des normes sur $K^{d+1}$ à  homothétie près s'identifie bijectivement (et de manière $G$-équivariante) à  l'espace topologique $|\BC\TC|$, ce qui permet de voir $\tau$ comme une application 
$$\tau: \H_K^d(C)\to |\BC\TC|.$$ Plus précisément, pour toute norme sur $|\cdot|$ sur $K^{d+1}$, on peut trouver une base $(f_i)_{0\le i\le d}$ de $K^{d+1}$ telle que
$|\sum_i a_i f_i|=\sup_i |a_i||f_i|$ pour tous $a_i$, et 
$ |\varpi|\le |f_{i+1}|\le |f_i|\le 1$ pour $i<d$.
Les boules $B(0,r)$ centrées en $0$ et de rayon $r$ forment une suite croissante de réseaux et on a $B(0,|\varpi|r)=\varpi B(0,r)$. Se donner les rayons minimaux de ces réseaux entre $|\varpi|$ et $1$ revient à se donner une famille de nombres $(t_i)_{0\le i\le k}\in ]0,1[^{k+1}$ dont la somme vaut $1$.  En particulier, on peut associer à la norme $|\cdot|$ le point $x$ de la cellule $\mathring{\sigma}$ de poids  $(t_i)_i$ avec $\sigma=\{B(0,r)\}_r$. De plus, $|\cdot|$ est déterminée par les réseaux $B(0,r)$.


\subsection{Les recouvrements par les tubes au-dessus des faces et des cellules\label{paragraphhdksimpfer}}

Soit $\sigma\in\BC\TC_{k}$ un simplexe de type $(e_0, e_1,\cdots, e_k)$ de présentation : \[M_0\supsetneq M_1\supsetneq\cdots\supsetneq M_k \supsetneq \varpi M_0.\] Nous chercherons à  décrire $$\H_{K,\sigma}^d:=\tau^{-1} (\sigma), \,\, \H_{K,\mathring{\sigma}}^d:=\tau^{-1} (\mathring{\sigma})$$ en suivant \cite[Section 6.3]{ds}. 

Comme dans \ref{paragraphbtgeosimpstd}, donnons-nous une base adaptée  $(f_0,\cdots, f_d)$ et les objets $N_i$, $R_i$ et $\tilde{R}_i$ qui s'en déduisent. Tous les vecteurs de $K^{d+1}$ seront écrits dans cette base. Si $x=\sum_{i} x_i f_i$ et $y=\sum_{i} y_i f_i$ on note $\langle x,y\rangle=\sum_{i} x_i y_i$.
Pour simplifier, nous allons écrire $l_i(z)=\langle z, f_{d_i}\rangle$. L'espace recherché $\H_{K,\sigma}^d(C)$ est l'ensemble des points $z\in\P^d_{  K}(C)$ pour lesquels l'ensemble des boules fermées de $\tau(z)$ sont contenues dans $\sigma$. Justifions l'égalité 
$$\H_{K,\sigma}^d=\{z\in \P^d_{ K} | \, \forall a\in M_i\backslash M_{i+1}, 1=|\frac{\left\langle z,a \right\rangle}{l_i(z)}|\,\text{et}\,\, |l_0(z)|\geq\cdots \geq |l_i(z)|\geq\cdots\geq |\varpi| |l_0(z)|\}.$$
Par construction, les éléments $f_{d_i}$ sont dans $M_i\backslash M_{i+1}$. Le premier jeu d'égalités garantit que tous les éléments de $M_i$  ne rencontrant pas $M_{i+1}$ ont même norme pour $\tau (z)$. La chaîne d'inégalités entraîne que les rayons des boules $M_i$ sont cohérents. On n'en déduit l'inclusion en sens direct. Réciproquement, prenons $z$ vérifiant les inégalités ci-dessus et intéressons-nous à une boule de rayon $r$ de $\tau(z)$ pour $|\varpi||l_0(z)|\le r< |l_0(z)|$. Il existe un entier $i$ pour lequel $|l_{i-1}(z)|>r\geq|l_i(z)|$. Deux cas de figure se présentent :
\begin{itemize}
\item Si $a$ est dans $M_i$, $r\geq|l_i(z)|\geq|\left\langle z,a \right\rangle|$.
\item  Sinon, $|\left\langle z,a \right\rangle|\geq|l_{i-1}(z)|>r$.
\end{itemize}
  Ainsi, $M_i=B(0,r)$ ce qui prouve l'inclusion en sens inverse.

 Par ultramétrie, il suffit de vérifier ces inégalités pour un système de représentants modulo $\varpi$ et $\OC_K^*$. Ainsi:
$$\H_{K,\sigma}^d=\{z\in \P^d_{ K} |\, \forall a\in R_i,\, 
1=|\frac{\left\langle z,a \right\rangle}{l_i(z)}|\,\text{et}\,\, |l_0(z)|\geq\cdots \geq |l_i(z)|\geq\cdots\geq |\varpi| |l_0(z)|\}.$$
Par finitude de $R_i$, l'espace ci-dessus est bien un ouvert rationnel affinoïde. 

 Soit $1\leq j\leq d$ un entier et soit $i_0$ l'unique entier tel que $d_{i_0}<j\leq d_{i_0+1}$. On pose     
 $$X_j=\frac{\left\langle z,f_j \right\rangle}{l_{i_0}(z)},\,\, X_0= \varpi\frac{l_0(z)}{l_k(z)}.$$ On a $| X_j | \le 1$ avec \'egalit\'e si $j \neq d_{i_0+1}$. On obtient un système de coordonnées $(X_0,\cdots,X_d)$ qui vérifie $\pro{X_{d_{i}}}{i=0}{k}=\varpi$. Pour tout $a$ les quantités $\frac{\left\langle z,a \right\rangle}{l_i(z)}$ s'expriment comme des polynômes $P_a (X_0,\cdots, X_d)$ à  coefficients dans $\OC_K$. En posant $P_\sigma=\pro{\pro{P_a}{a\in R_i}{}}{i=0}{k}$, on obtient la description suivante:

\begin{prop} On a 
$$\H_{K,\sigma}^d=\{(X_0,\cdots,X_d)\in \B^{d+1}_{K}|\,  \pro{X_{d_{i}}}{i=0}{k}=\varpi\,\, \text{et}\,\, 
|P_{\sigma}(X)|=1\}.$$
\end{prop}

On en déduit que $\H_{K,\sigma}^d$ admet un modèle entier $\H_{\OC_K,\sigma}^d =\spf  (\hat{A}_\sigma)$ ou $\hat{A}_\sigma$ est le complété $p$-adique de 

\begin{equation}\label{eqhatasigma}
\OC_K [X_0,\cdots, X_d,\frac{1}{P_\sigma}]/(\pro{X_{d_{i}}}{i=0}{k}-\varpi).
\end{equation}


 On obtient alors $ \H_{K,\sigma}^d=\spg (A_\sigma)$ avec $A_\sigma=\hat{A}_\sigma[\frac{1}{\varpi}]$. De même, la fibre spéciale est donnée par $ \H_{\F,\sigma}^d=\spec (\bar{A}_\sigma)$ avec \[\bar{A}_\sigma=\hat{A}_\sigma/\varpi=\F [X_0,\cdots, X_d,\frac{1}{\bar{P}_\sigma}]/(\pro{X_{d_{i}}}{i=0}{k}). \] 
 

Des arguments identiques fournissent 

$$\H_{K,\mathring{\sigma}}^d=\{z\in \P^d_{ K} |\,  \forall a\in R_i, \ 1=|\frac{\left\langle z,a \right\rangle}{l_i(z)}| \,\,\text{et}\, \, 
 |l_0(z)|>\cdots > |l_i(z)|>\cdots> |\varpi| |l_0(z)|\}$$
et l'on peut remplacer $R_i$ par $\tilde{R}_i$. 

 Considérons les affinoïdes

\[C_r= \{x=(x_1,\cdots, x_r)\in \B^r_K|\, \forall a\in \OC^{r+1}_K\backslash \varpi \OC^{r+1}_K, \  1=|\left\langle (1, x),a \right\rangle| \},\]
\[A_k=\{y=(y_1,\cdots, y_k)\in \B^k_K|\,  1>|y_1|>\cdots>|y_k|>|\varpi|\}\]
et les morphismes 
$$ \H_{K,\mathring{\sigma}}^d \rightarrow C_{e_i -1}, \,\, 
  [z_0,\cdots, z_d]  \mapsto  (\frac{ z_{d_{i}+1} }{ z_{d_{i}} },\frac{ z_{d_{i}+2} }{ z_{d_{i}} },\cdots, \frac{ z_{d_{i+1}-1} }{ z_{d_{i}} })\et $$
 $$\H_{K,\mathring{\sigma}}^d \rightarrow A_k,\,\, 
  [z_0,\cdots, z_d]  \mapsto  (\frac{ z_{d_{1}} }{ z_{d_{0}} },\frac{ z_{d_{2}} }{ z_{d_{0}} },\cdots, \frac{ z_{d_{k}} }{ z_{d_{0}} }).$$

\begin{prop}[\cite{ds} 6.4]\label{propdécompsimpouv} Les morphismes ci-dessus induisent un isomorphisme
\[ \H_{K,\mathring{\sigma}}^d \cong A_k \times \pro{C_{e_i-1}}{i=0}{k}\cong A_k \times C_\sigma.\]
\end{prop}

\section{Distributions et flèche d'intégration $\alpha$}

\subsection{Enoncé du résultat principal}

  Si $A$ est un groupe abélien, on note\footnote{Notons que $\tilde{\rm D}_1(A)$ coïncide avec la $G$-représentation ${\rm St}_1(A)$ mentionnée dans l'introduction pour $A=K$ ou $A=\Z/m\Z$ ($m$ premier à $p$) d'après \cite[§3, Theorem 1.]{scst} et \cite[lemme 3.2 + théorème 4.5]{iovsp}.}  
$$\tilde{\rm D}_1(A):=A\left\llbracket \HC\right\rrbracket^0:=\varprojlim_n A[\HC_n]^0,$$
où $A[\HC_n]^0$ est l'ensemble des $\mu\in  A[\HC_n]$ de masse totale nulle. L'action naturelle de 
$G$ sur $\HC$ induit une action de $G$ sur $\tilde{\rm D}_1(A)$.

Nous noterons $\H^d_{ L}:=\H^d_{K}\hat{\otimes} L$. Notre résultat principal s'énonce alors

\begin{theo}\label{theoprincunite}

Il existe un  isomorphisme $G$-\'equivariant \[\alpha : \tilde{\rm D}_1 (\Z) \simeq\Of^*(\H^d_{ L})/L^*.\]
\end{theo}



  Notre preuve met l'emphase sur les différentes
compatibilités entre les théories cohomologiques et s'inscrit dans la continuité des travaux de Schneider-Stuhler 
\cite{scst} et Iovita-Spiess \cite{iovsp}. 
Plus précisément, Iovita et Spiess ont construit des isomorphismes 
 \[\beta : \tilde{\rm D}_{1} (\Z/l\Z)\iso \het{1}(\H_{ C}^d, \mu_l)\et \gamma : \tilde{\rm D}_{1} (L)\iso \hdr{1} (\H_{ L}^d)\]
 en interprétant les éléments de $ \tilde{\rm D}_{1} (\Z/l\Z)$ et $\tilde{\rm D}_{1} (L)$ comme des mesures sur 
 $\HC^2$ et en intégrant des symboles en cohomologie étale et de de Rham (en fait de tels isomorphismes existent en tout degré cohomologique, comme nous allons le rappeler plus tard). En utilisant leur méthode nous montrons l'existence d'un morphisme $G$-équivariant 
$$\alpha :\tilde{\rm D}_1 (\Z)\to \Of^*(\H^d_{ L})/L^*$$
compatible avec les morphismes $\beta$ et $\gamma$ ie. les diagrammes suivant sont commutatifs
\[
\xymatrix{ 
\tilde{\rm D}_{1} (\Z) \ar[r]^-{\alpha} \ar[d] & \Of^*(\H_{ L}^d)/L^*\ar[d]^{\bar{\kappa}}  \\
\tilde{\rm D}_{1} (\Z/l\Z) \ar[r]^-{\beta} & \het{1}(\H_{ C}^d, \mu_l)  }\ \ \ \ \  
\xymatrix{ 
\tilde{\rm D}_{1} (\Z) \ar[r]^-{\alpha} \ar[d]^{} & \Of^*(\H_{ L}^d)/L^*\ar[d]^{d\log}  \\
\tilde{\rm D}_{1} (L) \ar[r]^-{\gamma} & \hdr{1} (\H_{ L}^d)  }
.\]
   En utilisant la théorie des résidus de de Shalit \cite{ds} nous montrons ensuite que 
   l'image de $\alpha$ est un facteur direct. La
compatibilité avec la cohomologie étale géométrique implique \ref{lemlinfty} que le supplémentaire de l'image de $\alpha$ est $l^\infty$-divisible modulo les constantes. Il s’agira alors d'établir par \ref{theoostrosen} qu’une fonction inversible  $l^\infty$-divisible dans $\Of^*(\H^d_{ L})/L^*$  est en fait constante.

\subsection{Cohomologie de $\H^d_C$, symboles et distributions}\label{paragraphstiovsp}

  Nous allons rappeler, suivant Iovita et Spiess \cite{iovsp} le calcul de la cohomologie \'etale 
  $l$-adique et de de Rham de $\H^d_C$ par intégration des symboles. Nous adapterons leur méthode dans le paragraphe suivant pour étudier les fonctions inversibles sur $\H^d_C$. 

  Soit $k\geq 0$ un entier.
 Rappelons que $\mathcal{H}$ désigne l'ensemble des hyperplans $K$-rationnels et que  $\mathcal{H}^{k+1}=\varprojlim_n \HC_n^{k+1}$ est un espace profini. Si $M$ est un $A$-module, on note $M^{\vee}={\rm Hom}_A(M,A)$ son $A$-dual algébrique. 
        L'ensemble des distributions sur $\HC^{k+1}$ à valeurs dans $A$, ie.  le dual algébrique des fonctions localement constantes sur $\HC^{k+1}$ à valeurs dans $A$, est
        défini par \[{\rm Dist}(\HC^{k+1},A):={\rm LC}(\HC^{k+1},A)^\vee\simeq\varprojlim_n A[\HC_n^{k+1}],\]
        les flèches de transition  $A[\HC_{n+1}^{k+1}]\to A[\HC_n^{k+1}]$ dans le système projectif étant induites par les projections naturelles  $\HC_{n+1}^{k+1}\to \HC_n^{k+1}$.
         L'isomorphisme ci-dessus provient du fait que la donnée d'une distribution est équivalente à celle d'une fonction des ouverts de $\HC^{k+1}$ dans $A$, additive sur les éléments  de la base de voisinage usuelle. De manière équivalente, cela revient à écrire 
         $ {\rm LC}(\HC^{k+1},A)=\varinjlim_{n} {\rm Hom}_{\rm Ens}(\HC_n^{k+1}, A)$ et à identifier ${\rm Hom}_{\rm Ens}(\HC_n^{k+1}, A)^{\vee}=A[\HC_n^{k+1}]$.
         
Pour tout vecteur unimodulaire $a$, on notera $l_a$, $l_a^{(n)}$ les formes linéaires associées  et $H_a= \ker(l_a)\in \HC$ (ou encore $H_a= \ker(l_a^{(n)})\in \HC_n$) les hyperplans associés.

 Si $M$ est un $A$-module, notons   $\Lambda_k (M)\subset {\rm LC}(\HC^{k+1},M)$ l'ensemble des fonctions localement constantes $f: \HC^{k+1}\to M$ vérifiant les deux relations   

\begin{enumerate}[label= \roman*)]
\item $f(H_{a_0},\cdots,H_{a_k})=0$ si $a_0,...,a_k$ sont linéairement dépendants. \label{reli}
\item $\som{(-1)^j f(H_{a_0},\cdots, \widehat{H_{a_j}},\cdots, H_{a_{k+1}}}{0\leq j\leq k+1}{})=0$ pour tous $H_{a_0}, \cdots, H_{a_{k+1}}$.\label{relii}
\end{enumerate}

On définit 
\[\tilde{\rm D}_k (A) =\Lambda_k (A)^\vee={\rm Dist}(\HC^{k+1},A)/{\rm Dist}(\HC^{k+1},A)_{deg}\] où ${\rm Dist}(\HC^{k+1},A)_{deg}$ est l'ensemble des distributions (dites dégénérées) s'annulant sur $\Lambda_k (A)$. On vérifie facilement que 
\[\tilde{\rm D}_k (A) =\varprojlim_n A[\HC_{n}^{k+1}]/I_{k,n} = \varprojlim_n \tilde{\rm D}_{k,n} (A) \] avec  $I_{k,n}$ l'ensemble des formes de $A[\HC_{n}^{k+1}]=\homm_{\ens}(\HC_n^{k+1}, A)^\vee$ qui s'annulent sur $\Lambda_k (A)\cap\homm_{\ens}(\HC_n^{k+1}, A)$ si l'on voit $\homm_{\ens}(\HC_n^{k+1}, A)\subset{\rm LC}(\HC^{k+1},A)$.

\begin{ex}
Dans les exemples ci-dessous d'éléments de $\Lambda_k (A)$, nous souhaitons mettre l'emphase sur les conditions \ref{reli} et \ref{relii} ci-dessus. Le caractère localement constant n'est pas évident et est démontré dans le lemme \ref{claimlc} qui va suivre (cf. aussi \cite{iovsp} pour le second et le troisième exemple). Soit $L$ une extension complète de $K$ et notons $\mathring{U}_{i,L}=\mathring{U}_{i} \hat{\otimes} L$.
\begin{enumerate}
\item[-] En posant $M=\Of^*( \mathring{U}_{i,L})/L^*\Of^{**}_m ( \mathring{U}_{i,L})$, l'application  \[f:(H_a, H_b)\mapsto \frac{l_a}{l_b}\]   est un élément de $\Lambda_1 (M)$. La condition \ref{reli} revient à  écrire $f(H,H)=1$. La condition \ref{relii} est équivalente à  la relation $\frac{l_a}{l_c}=\frac{l_a}{l_b}\frac{l_b}{l_c}$.
\item[-] En posant $M=\hdr{k} ( \mathring{U}_{i,L})$, l'application  \[g:(H_{a_0},\cdots, H_{a_k})\mapsto  d\log \frac{l_{a_1}}{l_{a_0}}\wedge d\log \frac{l_{a_2}}{l_{a_0}}\wedge \cdots \wedge d\log \frac{l_{a_k}}{l_{a_0}}\]  est un élément de $\Lambda_k (M)$.
\item[-] En posant $M=\het{k} ( \mathring{U}_{i,C}, \mu_l^{\otimes k})$ pour $l$ premier à  $p$ et en notant $\cup$ le cup-produit
et $\bar{\kappa}(f)\in \het{1}( \mathring{U}_{i,C}, \mu_l)$ l'image de $f\in \Of^*( \mathring{U}_{i,C})$ par l'application de Kummer, l'application
 \[h:(H_{a_0},\cdots, H_{a_k})\mapsto  \bar{\kappa} (\frac{l_{a_1}}{l_{a_0}})\cup \bar{\kappa} (\frac{l_{a_2}}{l_{a_0}})\cup \cdots \cup \bar{\kappa} (\frac{l_{a_k}}{l_{a_0}})\] est un élément de $\Lambda_k (M)$.
\end{enumerate}

\end{ex}

\begin{rem}\label{remident}

Quand $k$ vaut $1$, les modules $\tilde{\rm D}_{1,n}(A)$ ont une interprétation explicite. On peut les voir comme les ensembles $A[\HC_n]^0$ de fonctions $f$ de $A[\HC_n]$ de masse totale nulle. Ainsi, le choix d'un système de représentants $S_n$ de $\HC_n$ induit une identification pour tout $i$ \[\tilde{\rm D}_{1,n}(\Z)\cong \left\langle \frac{l_a}{l_b} :a\neq b\in S_n\right\rangle_{\Z\modut}\subset\Of^* (\mathring{U}_{i,L}).\] De même, on voit $\tilde{\rm D}_1(\Z)$ comme $\Z\left\llbracket \HC\right\rrbracket^0=\varprojlim_n \Z[\HC_n]^0$. 

\end{rem}

  Iovita et Spiess ont construit dans \cite{iovsp} des morphismes (pour $l\ne p$ premier)

 \[ \beta^{(k)} : \tilde{\rm D}_{k} (\Z/l\Z)\to \het{k} (\H_{ C}^d, \mu_l^{\otimes k}) \ {\rm et } \ \gamma^{(k)} : \tilde{\rm D}_{k} (L)\to \hdr{k} (\H_{ L}^d)\]
 à partir des symboles en cohomologie étale et de de Rham, i.e. à partir des fonctions $g$ et $h$ dans l'exemple ci-dessus. Plus précisément, ces morphismes sont uniquement caractérisés par le fait que 
 \[\gamma^{(k)}(\delta_{(H_{a_0},\cdots,H_{a_k})})= g(H_{a_0},\cdots,H_{a_k})\et \beta^{(k)}(\delta_{(H_{a_0},\cdots,H_{a_k})})=h(H_{a_0},\cdots,H_{a_k})\]
 où $\delta_{(H_{a_0},\cdots,H_{a_k})}$ est la masse de Dirac en $(H_{a_0},\cdots,H_{a_k})\in \HC^{k+1}$. 
  Nous allons écrire 
 $$\beta^{(k)}(\mu):=\int_{(H_{a_0},\cdots, H_{a_k})\in \HC^{k+1}} \bar{\kappa} (\frac{l_{a_1}}{l_{a_0}})\cup \cdots \cup \bar{\kappa} (\frac{l_{a_k}}{l_{a_0}}) {\rm d} \mu$$
 et de même avec $\gamma^{(k)}$. 
 
 Le théorème principal de \cite{iovsp} (voir lemme 3.2 et théorème 4.5 dans \textit{loc. cit}) est le suivant : 

\begin{theo}\label{theoiovsp}    
Les morphismes $\beta^{(k)}$ et $\gamma^{(k)}$ sont des isomorphismes $G$-équivariants \[ \beta^{(k)} : \tilde{\rm D}_{k} (\Z/l\Z)\simeq \het{k} (\H_{ C}^d, \mu_l^{\otimes k}) \ {\rm et } \ \gamma^{(k)} : \tilde{\rm D}_{k} (L)\simeq\hdr{k} (\H_{ L}^d).\]

\end{theo}

En fait, un résultat récent décrit la cohomologie étale à coefficients dans $\mu_N$ pour des entiers $N$ non nécessairement premier à $p$. Nous n'en aurons pas besoin pour  la suite. 

\begin{theo}[\cite{GPW3} Theorem 5.1.]\label{theocohopadiquehdk}
On a un isomorphisme pour tous entiers $N,k$ :
\[ \het{k}(\H_{ C}^d, \mu_N^{\otimes k}) \cong \tilde{\rm D}_k(\Z/N \Z).\]
\end{theo}

\subsection{Fonctions inversibles et distributions}
 
   Le résultat principal de cette section est le suivant. 
 
\begin{prop}\label{propintdist}

   Il existe un morphisme 
   $$\alpha : \tilde{\rm D}_1 (\Z) \to   \Of^*(\H_{ L}^d)/L^*$$
   uniquement caractérisé par 
   $$\alpha(\delta_{H_a, H_b})=\frac{l_a}{l_b}$$
   pour tous $H_a,H_b\in \HC$. De plus on dispose de diagrammes commutatifs 
 
 \[
\xymatrix{ 
\tilde{\rm D}_{1} (\Z) \ar[r]^-{\alpha} \ar[d] & \Of^*(\H_{ L}^d)/L^*\ar[d]^{\bar{\kappa}}  \\
\tilde{\rm D}_{1} (\Z/l\Z) \ar[r]^-{\beta} & \het{1}(\H_{ C}^d, \mu_l)  }\ \ \ \ \  
\xymatrix{ 
\tilde{\rm D}_{1} (\Z) \ar[r]^-{\alpha} \ar[d]^{} & \Of^*(\H_{ L}^d)/L^*\ar[d]^{d\log}  \\
\tilde{\rm D}_{1} (L) \ar[r]^-{\gamma} & \hdr{1} (\H_{ L}^d)  }
\]
avec $\beta=\beta^{(1)}$ et $\gamma=\gamma^{(1)}$.

\end{prop}

Comme ci-dessus, nous allons écrire

\[\alpha(\mu):=\int_{(H_{a}, H_{b})\in \HC^{2}}   \frac{l_{a}}{l_{b}}{\rm d} \mu\in  \Of^*(\H_{ L}^d).\]

\begin{proof}   Le point crucial est le lemme suivant:

\begin{lem}\label{claimlc}

Si $(H_{a_1}, H_{b_1})$ et $(H_{a_2}, H_{b_2})$ sont dans $\HC^2$ et vérifient $(H_{a_1}, H_{b_1})\equiv (H_{a_2}, H_{b_2})\pmod{\varpi^n}$ alors pour $i<n$, \[\frac{l_{a_1}}{l_{b_1}}\frac{l_{b_2}}{l_{a_2}}\in L^*\Of^{**}_{n-i} ( \mathring{U}_{i,L}).\]

\end{lem}

\begin{proof}
Prenons $\tilde{a}$, un vecteur unimodulaire tel que $a_1 =\lambda (a_2 +\varpi^{n'}\tilde{a})$ avec $n'\geq n> i$ et $\lambda\in \OC_K^*$, alors $$\frac{l_{a_1}}{l_{a_2}}= \lambda(1+\varpi^{n' -i}(\varpi^i \frac{l_{\tilde{a}}}{l_{a_2}})).$$ Mais $\varpi^i \frac{l_{\tilde{a}}}{l_{a_2}}\in \Of^{++} ( \mathring{U}_{i,L})$ par définition de $\mathring{U}_{i,L}$ (cf paragraphe \ref{paragraphhdksimpfer}) d'où $\frac{l_{a_1}}{l_{a_2}}\in  L^*\Of^{**}_{n-i} ( \mathring{U}_{i,L})$. On raisonne de même pour $(b_1, b_2)$ et on conclut par stabilité multiplicative de  $ L^*\Of^{**}_{n-i} ( \mathring{U}_{i,L})$.
\end{proof}

Remarquons les isomorphismes 
  $$\Of^*(\H_{ L}^d)/L^* \cong \varprojlim_{i} \Of^*( \mathring{U}_{i,L})/L^*, \quad \Of^*( \mathring{U}_{i,L})/L^*\simeq \varprojlim_m\Of^*( \mathring{U}_{i,L})/L^*\Of^{**}_m ( \mathring{U}_{i,L}).$$ 

Fixons, pour tout $n$, $S_n$ et $S'_n$ des systèmes de représentants de $\HC_n$. Considérons la flèche
\begin{align*}
\tilde{\alpha}_{n,i} : \Z[\HC_n^2]  &\rightarrow  \Of^*( \mathring{U}_{i,L})/L^*  &\\
  \delta_{(H_a,H_b)} & \mapsto  \frac{l_a}{l_b} & a,b\in S_n,
\end{align*} 
elle induit par passage au quotient une application  $\alpha_{n,i}: \tilde{\rm D}_{1} (\Z)\to 
\Of^*( \mathring{U}_{i,L})/L^*$ donnée par \[\alpha_{n,i}(\mu)=\prod_{a,b\in S_n}(\frac{l_a}{l_b})^{\mu_{a,b}}\in \Of^*( \mathring{U}_{i,L})/L^*\] avec $\mu=(\mu_n)_n:=(\sum_{a,b\in \HC_n}\mu_{a,b}\delta_{H_a,H_b})_n\in\limp_n \tilde{\rm D}_{1,n} (\Z)=\tilde{\rm D}_{1} (\Z) $. On définit de même une flèche $\alpha'_{n,i}$ grâce aux systèmes de représentants $S'_n$. 

Pour $n'>n>i+1$, on observe les congruences suivantes grâce aux compatibilités des $\mu_n$ et au lemme \ref{claimlc} :
\begin{equation}\label{eq:congint1}
\alpha_{n',i}(\mu)\equiv\alpha_{n,i}(\mu)\equiv\alpha'_{n,i}(\mu)\pmod{L^*\Of^{**}_{n -i}( \mathring{U}_{i,L})},
\end{equation}
\begin{equation}\label{eq:congint2}
\alpha_{n,i+1}(\mu)|_{\mathring{U}_{i,L}}\equiv\alpha_{n,i}(\mu)\pmod{L^*\Of^{**}_{n -i}( \mathring{U}_{i,L})}.
\end{equation}
Ainsi, les deux suites $(\alpha_{n,i}(\mu))_n$ et $(\alpha'_{n,i}(\mu))_n$ sont de Cauchy et convergent vers la même valeur $\alpha_{i}(\mu)$ d'après \eqref{eq:congint1}. D'après \eqref{eq:congint2},  $\alpha_{i}(\mu)$ est la restriction de $\alpha_{i+1}(\mu)$ sur $\mathring{U}_{i,L}$ ce qui détermine $\alpha(\mu)\in \Of^*(\H_{ L}^d)/L^*$. 

Intéressons nous à l'image par $\alpha$ d'un Dirac $\delta_{H_a, H_b}$. Comme l'application $\alpha$ ne dépend pas du choix du système de représentant $(S_n)_n$, on le choisit de telle manière à ce  que $H_a$ et $H_b$ sont dans $S_n$ pour tout $n$ assez grand. Par définition, on a la relation $\alpha_{n,i}(\delta_{H_a, H_b})=\frac{l_a}{l_b}$ pour tout $i$ et pour $n$ assez grand. On en déduit l'égalité   $\alpha(\delta_{H_a, H_b})=\frac{l_a}{l_b}$ par convergence. L'unicité de $\alpha$ provient alors  de la densité  du module engendré par les Dirac.

Passons à l'application $\gamma$ (le procédé est le même pour $\beta$).   Puisque la série formelle du logarithme converge sur $\Of^{**} ( \mathring{U}_{i,L})$, on a 
$L^*\Of^{**} ( \mathring{U}_{i,L}) \subset \ker (d\log)$. Ainsi d'après \ref{claimlc}, pour tout $i>0$ et $\mu=(\sum_{a,b\in S_n}\mu_{a,b}\delta_{H_a,H_b})_n\in \tilde{\rm D}_{1} (L) $, la somme suivante ne dépend pas  de $n$ pour $n>i$ :
\[\gamma_i(\mu):=\sum_{a,b\in S_n}\mu_{a,b}d\log (\frac{l_a}{l_b})\in \hdr{1} ( \mathring{U}_{i,L}).\]
 De plus, si les éléments $\mu_{a,b}$ ci-dessus sont entiers, on a la relation  $\gamma_i(\mu)=d\log (\alpha_{n,i}(\mu))$ pour $n>i$. On vérifie aussi que $\gamma_i(\mu)$ est la restriction de $\gamma_{i+1}(\mu)$  sur $\mathring{U}_{i,L}$ ce qui détermine une application $\gamma :\tilde{\rm D}_{1} (L) \to \hdr{1} (\H_{ L}^d)$.

Pour la compatibilité avec le morphisme $\alpha$, on observe la congruence \[\alpha (\mu)|_{\mathring{U}_{i,L}}\equiv\alpha_{n,i}(\mu) \pmod{L^*\Of^{**}_{n -i}( \mathring{U}_{i,L})}\] pour $n>i$ et $\mu\in \tilde{\rm D}_{1} (\Z)$. Toujours par convergence de la série formelle du logarithme sur $\Of^{**}_{n -i}( \mathring{U}_{i,L})$, on en déduit la suite d'égalité  pour $n>i$\[d\log(\alpha(\mu))|_{\mathring{U}_{i,L}}=d\log(\alpha_{n,i}(\mu))=\gamma_i(\mu)=\gamma(\mu)|_{\mathring{U}_{i,L}}.\]
On a alors $d\log(\alpha(\mu)) =\gamma(\mu)$.

\end{proof} 

\begin{rem}

Dans la preuve, nous nous sommes servis d'une famille de systèmes de représentants $(S_n)_n$ des ensembles finis $(\HC_n)_n$ par des vecteurs unimodulaires dans $\OC_K^{d+1}$ pour construire la flèche $\alpha$ et nous avons montré qu'elle n'en dépendait pas.

Si l'on impose de plus à nos systèmes de représentants qu'ils vérifient $S_n\subset S_{n+1}$ et $a-b\in \varpi^n\OC_K^{d+1}$  pour tout élément $a,b\in S_{n+1}$ dont les projections dans $\HC_n$ coïncident (possible en procédant aux bonnes renormalisations), alors le procédé exhibé dans la preuve précédente permet de construire directement une flèche $ \tilde{\rm D}_1 (\Z) \to   \Of^*(\H_{ L}^d)$ telle l'image de $\delta_{H_a, H_b}$ est $\frac{l_a}{l_b}$ pour $a,b\in S_{n}$. Cette dernière fournirait alors un scindage à la suite exacte du théorème principal \ref{theoprincunite} qui dépend de la famille $(S_n)_n$. En particulier, la décomposition en produit qui en découle n'est pas $\gln_{d+1}(\OC_K)$-équivariante (et encore moins $G$-équivariante) car l'action de ce groupe permute les systèmes de représentants $(S_n)_n$. 

Pour une description plus canonique de $\Of^*(\H_{ L}^d)$ en tant qu'extension de $L^*$ et de $\tilde{\rm D}_1 (\Z)$ qui prend en compte l'action de $G$, nous renvoyons à l'article récent \cite{geh}.

\end{rem}

    Nous allons étudier l'application $\alpha$ dans la section suivante.
    
    \section{Propriétés de $\alpha$ et fin de la preuve} 
    
    \subsection{Résidus et cochaines harmoniques}
    
    Le but de ce paragraphe est de démontrer que ${\rm im}(\alpha)$ est un facteur direct de $\Of^*(\H_{ L}^d)/L^*$. Cela repose sur 
    le diagramme \[\xymatrix{ 
\tilde{\rm D}_{1} (\Z) \ar[r]^-{\alpha} \ar[d]^{} & \Of^*(\H_{ L}^d)/L^*\ar[d]^{d\log}  \\
\tilde{\rm D}_{1} (L) \ar[r]^-{\gamma} & \hdr{1} (\H_{ L}^d)  }
,\] mais nous avons besoin du raffinement suivant : 

\begin{prop}\label{intres}
On a un diagramme commutatif 
\[
\xymatrix{
\tilde{\rm D}_{1} (\Z) \ar[r]^-{\alpha} \ar[rd]^{\rm Id} &\Of^*(\H_{ L}^d) /L^* \ar[d]^{\tilde{\gamma}} \\
  & \tilde{\rm D}_{1} (\Z)
  }
  \]
  o\`u $\tilde{\gamma}$ est induite par $\gamma^{-1} \circ d \log$.
\end{prop}
 
   La preuve de ce résultat est basée sur la théorie des résidus utilisée par de Shalit dans 
   \cite{ds} pour calculer la cohomologie de de Rham des espaces de Drinfeld. Nous allons donc commencer par rappeler ses constructions. Si $A$ est un groupe abélien, on note 
   $\CC_{\rm har}^k (A)$ l'ensemble des cochaines harmoniques d'ordre $k$ à  valeurs dans $A$. Il s'agit d'applications  $c : \widehat{\BC\TC}_{k} \to A$ vérifiant certaines conditions d'harmonicité, reliant la valeur d'un simplexe pointé à celle des simplexes contigus (plus une relation entre les différents pointages d'un même simplexe). Nous renvoyons à  \cite[ paragraphe 3.1]{ds} pour l'énoncé explicite de ces conditions (ces relations y sont données sur $K$, mais les combinaisons linéaires sont à coefficients dans $\Z$ et peuvent être vues dans n'importe quel groupe abélien).  

Nous allons maintenant construire un morphisme  $\hdr{k} (\H_{ L}^d)\to \CC_{\rm har}^k (L)$. 
Nous renvoyons le lecteur aux paragraphes \ref{paragraphbtgeosimpstd} et  \ref{paragraphhdksimpfer} pour les notations à suivre.
Fixons un simplexe pointé $\sigma\in\widehat{\BC\TC}_{k}$ : \[M_0\supsetneq M_1\supsetneq\cdots\supsetneq M_k \supsetneq \varpi M_0\]  de type $(e_0, e_1,\cdots, e_k)$ et donnons-nous une base adaptée  $(f_0,\cdots, f_d)$ (cf \ref{paragraphbtgeosimpstd}). Le choix de cette base détermine un isomorphisme 
\[ \H_{K,\mathring{\sigma}}^d \cong A_k \times \pro{C_{e_i-1}}{i=0}{k}\cong A_k\times C_\sigma.\]
On notera encore $A_k$ et $C_\sigma$ les changements de base à $L$.

On a une notion naturelle de r\'esidu sur $A_k$: si $\omega= \som{a_\nu z^\nu d \log(z_1)\wedge \cdots \wedge d \log(z_k)}{}{}  \in \Omega^k_{A_k}$ (avec les bonnes conditions de convergence sur les $a_\nu\in L$), alors ${\rm res}_{A_k} (\omega)= a_0$. \'Etendons-la \`a $\Omega^k_{\H_{ L,\mathring{\sigma}}^d}$ en écrivant 
\[ \Omega^k_{\H_{ L,\mathring{\sigma}}^d} =\drt{\Omega^s_{A_k} \widehat{\otimes}\Omega^t_{C_\sigma}}{s+t=k}{}   \] et en définissant ${\rm res}_\sigma : \Omega^k_{\H_{ L,\mathring{\sigma}}^d} \to \Omega^k_{A_k} \widehat{\otimes} \Of_{C_\sigma} \fln{{\rm res}_{A_k} \widehat{\otimes} {\rm Id}_{\Of_{C_\sigma}}}{} \Of_{C_\sigma}$. 
Par calcul direct, pour toute forme ferm\'ee $\omega$  dans $\Omega^k_{\H_{ L,\mathring{\sigma}}^d} $, ${\rm res}_\sigma(\omega)$ est une $0$-forme ferm\'ee donc un \'el\'ement de $L$. De m\^eme, on v\'erifie que ${\rm res}_\sigma(\omega)=0$ si $\omega$ est exacte et  que le résidu ne dépend pas de la décomposition en produit du tube au-dessus de $\mathring{\sigma}$ ni du choix de la base adaptée au simplexe $\sigma$. Cela d\'efinit l'application 
\[ {\rm res} : \hdr{k}(\H_{ L}^d) \to \CC^k_{\text{har}}(L), \, \omega\mapsto ( \sigma \mapsto {\rm res}_\sigma(\omega|_{\H_{ L,\mathring{\sigma}}^d})).\]
En effet, pour $\omega$ une $k$-forme fermée, ${\rm res}(\omega)$ vérifie bien les conditions d'harmonicité \cite[théorème 7.7]{ds}. Le r\'esultat principal de \cite{ds} montre 

\begin{theo}\label{theods}  \cite[ théorème 8.2 ]{ds}
Le morphisme ${\rm res}$ est un isomorphisme.

\end{theo}

Nous allons expliciter l'isomorphisme $\phi$ rendant commutatif le diagramme 
\[\xymatrix{ 
\hdr{k}(\H_{ L}^d) \ar[dr]^{\rm  res} &    \\
\tilde{\rm D}_{k} (L) \ar[u]^{\gamma^{(k)}}\ar[r]^-{\phi}_{\sim} & \CC^k_{\text{har}}(L) }.
\]
Pour tout  
$\sigma\in \widehat{\BC\TC}_k$, on définit  la fonction localement constante $\lambda_\sigma \in \Lambda_k (\Z)$ par  \[(H_{a_0},\cdots, H_{a_k})\mapsto {\rm res}_\sigma(d\log \frac{l_{a_1}}{l_{a_0}}\wedge d\log \frac{l_{a_2}}{l_{a_0}}\wedge \cdots \wedge d\log \frac{l_{a_k}}{l_{a_0}}|_{\H_{ L,\mathring{\sigma}}^d})\in\{-1,0,1\}.\] La valeur de $\lambda_\sigma$ en un $(k+1)$-uplet d'hyperplans est déterminée par l'arrangement de ces hyperplans par rapport aux modules $M_i$ définissant le simplexe $\sigma$ (cf \cite[ paragraphe 1.3]{dsal}). L'application recherchée est donnée par  

\[\phi : \mu \mapsto ( \sigma \mapsto \mu (\lambda_\sigma )).\] Pour vérifier la commutativité,  il suffit de travailler sur les masses de Dirac et cela découle de la construction des $\lambda_\sigma$. Comme les $\lambda_\sigma$ vivent dans $\Lambda_k (\Z)$, on peut construire pour tout groupe abélien un morphisme  $\tilde{\rm D}_{k} (A)\to \CC^k_{\text{har}}(A)$. En fait, on a le résultat remarquable : 

\begin{theo} \cite[ corollaire 2.2]{dsal} \label{theodsal}
On a
\[\Lambda_k (\Z)=\left\langle \lambda_\sigma, \sigma\in \widehat{\BC\TC}_k \right\rangle_{\Z\modut}.\]

\end{theo}


\begin{coro}

La flèche $\phi : \tilde{\rm D}_{k} (\Z)  \to \CC^k_{\rm har}( \Z)$ est un isomorphisme.

\end{coro}

\begin{proof}
Soit $c\in \CC^k_{\text{har}}( \Z)\subset\CC^k_{\text{har}}( L)$ et $\mu=\phi^{-1}(c) \in \tilde{\rm D}_{k} (L)$ la mesure correspondante par \ref{theods}. On veut $\mu \in \tilde{\rm D}_{k} (\Z)$. Il suffit d'observer l'équivalence suivante qui est conséquence directe du théorème précédent
\[\forall \sigma\in \widehat{\BC\TC}_k, \mu (\lambda_\sigma)\in\Z \Leftrightarrow \forall \lambda\in \Lambda_k (\Z), \mu (\lambda)\in\Z.\]
\end{proof}


 
 

   Revenons maintenant à la preuve de la proposition \ref{intres}. 
D'après la discussion précédente il suffit de montrer pour toute fonction inversible $u$ et tout $\sigma\in \widehat{\BC\TC}_1$ de type $(e_0,e_1)$, ${\rm res}_\sigma(d \log(u)) \in \Z$. Puisque $A_1$ est une couronne ouverte, on a d'après \cite[Lemme 4.4.]{J1} une égalité 
\[\Of^*(\H_{ L,\mathring{\sigma}}^d) = \Of^* (C_{\sigma}\times A_1)= \Of^* (C_{\sigma}) \Of^{**}(\H_{ L,\mathring{\sigma}}^d) \times z^{\Z},\] où $z$ est la variable de $A_1$. Comme 
$u|_{ \H_{ L,\mathring{\sigma}}^d}$ est dans $\Of^*(\H_{ L,\mathring{\sigma}}^d)$, on peut donc décomposer $u=g(1+h)z^{\alpha}$ comme ci-dessus. On a
\[ d \log(u)= d \log(g) + d \log(1+h) + \alpha d \log(z).\] 
Or $d \log(g) \in \Omega^1_{C_{\sigma}}$, son r\'esidu est donc nul. La forme $d \log(1+h)$ est exacte (puisque $\log(1+h)$ est bien une fonction analytique), son r\'esidu est donc nul. Donc ${\rm res}_\sigma(d \log(u))= \alpha \in \Z$, ce qui permet de conclure.



\subsection{Cohomologie \'etale arithmétique}

Nous aurons besoin du premier degré de la cohomologie étale $l$-adique arithmétique de $\H_{ L}^d$.

\begin{prop}\label{diagfund}

Soit $l$ un entier premier à  $p$ (non nécessairement premier), $L\subset C$ une extension complète de $K$,  on a \[\kappa(\widetilde{{\rm im}(\alpha)})=\het{1}(\H_{ L}^d, \mu_l)\] où $\widetilde{{\rm im}(\alpha)}\subset \Of^*(\H_{ L}^d)$ est la préimage de ${{\rm im}(\alpha)}\subset \Of^*(\H_{ L}^d)/L^*$

\end{prop}

\begin{proof} Cela découle de la compatibilité entre $\alpha$ et l'application de Kummer, du calcul de la cohomologie étale géométrique \ref{theoiovsp} et 
du point  technique suivant: 

\begin{lem}\label{theoh1etarithm}

Soit $l$ un entier premier à  $p$ (non nécessairement premier), $X$ un $L$-espace analytique géométriquement connexe et $\bar{X}=X\hat{\otimes}C$. Supposons que  $\bar{\kappa} (\Of^*(X))=\het{1} (\bar{X},\mu_l)$. Alors, on a une suite exacte courte \[0\to L^*/(L^*)^l\to \het{1} (X,\mu_l)\to\het{1}(\bar{X}, \mu_l)\to 0.\]
Si, de plus, il existe un sous-groupe $H\subset \Of^*(X)$ $l$-saturé dans $\Of^*(\bar{X})$ vérifiant $\bar{\kappa} (H)=\het{1} (\bar{X},\mu_l)$, cette suite est scindée.
\end{lem}

\begin{proof}
La suite spectrale de Hochschild-Serre (où $\GC_L$ est le groupe de Galois absolu du corps $L$) \[E_2^{r,s} :\hgal{r} (\GC_L,\het{s}(\bar{X}, \mu_l))\Rightarrow \het{r+s} (X, \mu_l)\] induit une suite exacte \[0\to\hgal{1} (\GC_L,\mu_l(\bar{X}))\to \het{1}(X, \mu_l)\to \het{1}(\bar{X}, \mu_l)^{\GC_L}.\] 
  
   Considérons le diagramme commutatif

\[\xymatrix{ 
  & \het{1}(X, \mu_l) \ar[d]^{} \\
\Of^* (X) \ar[ru]^{\kappa}\ar[r]^-{\bar{\kappa}} & \het{1}(\bar{X}, \mu_l)^{\GC_L}}.
\]
L'application $\bar{\kappa}$ est $\GC_L$-équivariante et surjective (par hypothèse). Puisque 
$\Of^*(X)\subset \Of^*(\bar{X})^{\GC_L}$, on en déduit les égalités  $\het{1}(\bar{X}, \mu_l)^{\GC_L}=\het{1}(\bar{X}, \mu_l)=\bar{\kappa}(\Of^*(X))$, ainsi que la 
 surjectivité de $ \het{1}(X, \mu_l)\to \het{1}(\bar{X}, \mu_l)=\het{1}(\bar{X}, \mu_l)^{\GC_L}$.

On a la suite d'isomorphismes \[ \hgal{1} (\GC_L,\mu_l(\bar{X}))\cong \hgal{1} (\GC_L,\mu_l(C))\cong L^*/(L^*)^l.\] 
La première égalité provient de la connexité géométrique de $X$, la deuxième se déduit de la suite exacte de Kummer et du théorème de Hilbert 90. On en déduit la suite exacte courte \[0\to L^*/(L^*)^l\to \het{1} (X,\mu_l)\to\het{1}(\bar{X}, \mu_l)\to 0.\] 
 
Supposons maintenant l'existence du sous-groupe  $H$. Par hypothèse de surjectivité de $\bar{\kappa}_{|H}$, on obtient un diagramme commutatif où toutes les flèches sont surjectives : \[
\xymatrix{
\kappa (H) \ar[r]^{}  &\het{1}(\bar{X}, \mu_l)  \\
H/H^l \ar[u]^{\kappa} \ar[ru]_{ \bar{\kappa}}& 
  }.
  \] Par hypothèse de saturation sur $H$ la flèche diagonale est injective. Ainsi, les trois flèches du diagramme précédent sont bijectives, ce qui permet de scinder la suite exacte. 
\end{proof}

  \end{proof}
  
  \subsection{Preuve du théorème \ref{theoprincunite}\label{sssectionproofosthdk}}
  
  Nous allons prouver le théorème \ref{theoprincunite}. Prenons $l$ un entier premier à  $p$.
  Reprenons les notations dans la proposition \ref{intres}. Cette proposition montre que 
   $\ker(\tilde{\gamma})$ est un suppl\'ementaire de ${\rm im}(\alpha) \cong \tilde{\rm D}_{1} (\Z)$ dans $\Of^*(\H^d_{ L})/L^*$. On veut montrer que $\ker(\tilde{\gamma})=0$. Nous commençons par: 

\begin{lem}\label{lemlinfty}

$\ker(\tilde{\gamma})$ est $l$-divisible.

\end{lem}

\begin{proof}

Soit $f$ dans $\ker(\tilde{\gamma})$ que l'on relève en $\tilde{f}\in \Of^*(\H^d_{ L})$. D'apr\`es la proposition \ref{diagfund}, on peut trouver $\tilde{g}$ dans $\widetilde{{\rm im}(\alpha)}$ tel que $\kappa (\tilde{g})=\kappa (\tilde{f})$. Alors $\tilde{f}\tilde{g}^{-1}$ 
admet une racine $l$-ième dans $\Of^*(\H^d_{ L})$ et donc une racine  dans $\Of^*(\H^d_{ L})/L^*=\ker(\tilde{\gamma})\oplus{\rm im}(\alpha)$ que l'on note $u$. L'élément $f$ est la composante de $u^l$ dans $\ker(\tilde{\gamma})$ donc, en projetant $u$ dans $\ker(\tilde{\gamma})$, on obtient une racine $l$-ième de $f$ dans $\ker(\tilde{\gamma})$.

\end{proof}

Il suffit donc de montrer que les fonctions inversibles $f$ de $\H_{ L}^d$, $l^\infty$-divisibles (modulo les constantes), sont en fait constantes. Pour cela, on la projette dans $\Of^*(\H_{ L, \sigma}^d)/L^*\Of^{**} (\H_{ L, \sigma}^d)$ qui est de type fini pour tout simplexe $\sigma$ de l'immeuble $\BC\TC$ d'après la proposition \ref{theoostrosen} ci-dessous (combiné avec les résultats du paragraphe \ref{paragraphhdksimpfer}). Mais $f$  est encore $l^\infty$-divisible dans $\Of^*(\H_{ L, \sigma}^d)/L^*\Of^{**} (\H_{ L, \sigma}^d)$ pour tout simplexe $\sigma$ d'où $\forall \sigma, f\in L^*\Of^{**} (\H_{ L, \sigma}^d)$.  

On écrit alors $f=\lambda_\sigma (1+g_\sigma)\in L^*\Of^{**} (\H_{ L, \sigma}^d)$. On a en particulier $\Vert f\Vert_{\H^d_{ L,\sigma}}=|\lambda_\sigma|$ pour tout $\sigma$. Soit maintenant deux simplexes $\sigma_1$ et $\sigma_2$ s'intersectant, alors $f=\lambda_{\sigma_1}(1+g_{\sigma_1})$ est encore une décomposition dans $L^*\Of^{**} (\H_{ L, \sigma_1\cap\sigma_2}^d)$ et on a \[\Vert f\Vert_{\H^d_{ L,\sigma_1}}=|\lambda_{\sigma_1}|=\Vert f\Vert_{\H^d_{ L,\sigma_1\cap\sigma_2}}=|\lambda_{\sigma_2}|=\Vert f\Vert_{\H^d_{ L,\sigma_2}}.\] Ainsi par connexité, la quantité $\Vert f\Vert_{\H^d_{ L,\sigma}}$ ne dépend pas de $\sigma$. D'où $\Vert f\Vert_{\H^d_{ L}}=\Vert f\Vert_{\H^d_{ L,\sigma}}<\infty$. La fonction $f$ est bornée donc constante d'après \cite[lemme 3]{ber5}.

\begin{prop}\label{theoostrosen}

Soit $P\in \OC_K[X_0,\cdots, X_d]$ non nul, 
 $Y=\spg(A)$ un affinoïde sur $L$ de la forme \[A=L\langle X_0,\cdots, X_d,  \frac{1}{P}\rangle /(X_0\cdots X_r-\varpi)\] alors le groupe $\Of^*(Y)/L^*\Of^{**} (Y)$ est un $\Z$-module de type fini.

\end{prop}

\begin{proof}
Quitte à renormaliser le polynôme $P$ et à prendre $r$ minimal, on peut supposer que les éléments $\varpi$ et $(X_i)_{0\le i\le r}$ ne divisent pas $P$. Notons $\hat{A}$ le compl\'et\'e $p$-adique de $\OC_L[ X_0,\cdots, X_d,  \frac{1}{P}] /(X_0\cdots X_r-\varpi)$ et $\bar{A}=\hat{A}/\mG_L \hat{A}$. 

L'argument consiste à établir les  résultats suivant  :
\begin{itemize}
\item L'inclusion $\hat{A}\subset A$ induit une bijection $ \hat{A}^*/ \OC_L^*(1+ \mG_L \hat{A})\to A^*/ (L^*(1+ \mG_L \hat{A})\prod_{i=0}^{r-1} X_i^{\Z})$.
\item  $\bar{A}^*/ \kappa^*$ est  de type fini sur $\Z$.
\end{itemize} 
Notons l'isomorphisme $\hat{A}^*/ \OC_L^*(1+ \mG_L \hat{A})\cong \bar{A}^*/ \kappa^*$. En particulier, ces deux faits montrent que $A^*/ (L^*(1+ \mG_L \hat{A})\prod_{i=0}^{r-1} X_i^{\Z})$ est de type fini. Mais c'est un quotient de $A^*/ L^*(1+ \mG_L \hat{A})$ par un sous-groupe de type fini ce qui conclut la preuve.

\'Etablissons le premier point. L'injectivité est claire et concentrons-nous sur la surjectivité. Cela revient à prouver que toute fonction inversible $u\in A^*$ est dans $\hat{A}^* L^*\prod_{i=0}^{r-1} X_i^{\Z}$. Nous avons besoin de quelques méthodes  de \cite{ber7}.
La composée suivante est un morphisme pluri-nodal au sens de \cite[Definition 1.1]{ber7}
\[\spf (\hat{A})\to \spf (\OC_L\langle X_0,\cdots, X_d  \rangle /(X_0\cdots X_r-\varpi))\to \spf (\OC_L)\]
car le premier morphisme est étale et le second est poly-stable (voir \cite[Definition 1.2]{ber7}). D'après \cite[Proposition 1.4]{ber7}, on a $\hat{A}=\Of^+(Y)$ et la norme spectrale prend ses valeurs dans $|L^*|$. Pour $u$ dans $A^*$, on peut alors trouver une constante inversible $\lambda$ telle que $u/\lambda \in \hat{A}\backslash\mG_L\hat{A}$. 

La fibre sp\'eciale $\spec(\bar{A})$ de $\spf (\hat{A})$ a exactement $r+1$ composantes irr\'eductibles $(V(X_i))_i$ et notons $\alpha_i$ l'ordre d'annulation de $u/\lambda$ sur  $V(X_i)$. Comme  $u/\lambda \neq 0$ dans $\bar{A}$, il existe $i_0$ tel que $\alpha_{i_0}=0$. Alors $X^{\alpha} \in \hat{A}$  a m\^eme ordre d'annulation que $u/\lambda$ sur chaque composante irr\'eductible et  $\frac{u}{\lambda X^\alpha}\in\hat{A}^*$ ce que  l'on voulait établir. 

Passons au deuxième point et étudions $\bar{A}^*/ \kappa^*$. Une fonction en fibre sp\'eciale est déterminée par sa restriction sur chacune des composantes irréductibles. De plus, elle est inversible  si et seulement si chacune de ces restrictions l'est. On a alors une injection \[\bar{A}^*/ \kappa^*\flinj\prod_{i=0}^{r}\Of^*(V(X_i))/\kappa^* \] et chaque terme est de la forme : \[ \Of^*(V(X_i)) \cong \kappa^* \times \prod_{j} (P^{(i)}_j )^{\Z}\] o\`u $P(X_0, \dots, X_{i-1}, 0, X_{i+1}, \dots , X_d)= \prod_j (P^{(i)}_j)^{\alpha_{i,j}}$ est la d\'ecomposition en produit d'irr\'eductibles dans\footnote{Notons que $V(X_i)$ est affine et $\Of(V(X_i))=\kappa[X_0, \dots, X_{i-1}, X_{i+1}, \dots , X_d, \frac{1}{P(X_0, \dots, X_{i-1}, 0, X_{i+1}, \dots , X_d)}]$.} $\kappa[X_0, \dots, X_{i-1}, X_{i+1}, \dots , X_d]$. Le résultat en découle.
\end{proof}

\begin{rem}\label{remostfin}
Si $X=\spg(A)$ est un $L$-affinoïde géométriquement connexe et géométriquement normal, avec $L$ une extension finie de $K$ (ou plus généralement si la ramification de $L/K$ est finie)  alors $\Of^*(X)/L^*\Of^{**} (X)$ est un $\Z$-module de type fini.  La preuve découle des arguments de \ref{theoostrosen} quand $L=K$. Dans le cas particulier de \ref{theoostrosen}, on peut facilement relier les fonctions inversibles pour $L$ quelconque au cas où $L=K$, et obtient ce résultat de finitude sur n'importe quel corps de base (pas forcément local).
\end{rem}

\nocite{dr1}
\newpage

\bibliographystyle{alpha}
\bibliography{inv_v1}
\end{document}